\documentclass{article}

\usepackage[left=1.25in,top=1.25in,right=1.25in,bottom=1.25in,head=1.25in]{geometry}
\usepackage{amsthm}
\makeatletter
\def\th@plain{%
  \thm@notefont{}
  \itshape 
}
\def\th@definition{%
  \thm@notefont{}
  \normalfont 
}
\makeatother
\usepackage{amsmath,amssymb,mathtools}
\usepackage{verbatim,float,url,enumerate}
\usepackage{graphicx,subfigure,psfrag}
\usepackage{natbib,xcolor}
\usepackage{microtype,hyperref}
\usepackage{breqn} 
\usepackage[english]{babel}
\usepackage[nottoc]{tocbibind}
\usepackage{hyperref}
\hypersetup{
    colorlinks=true, 
    linktoc=all,     
    linkcolor=blue,  
}
\newtheorem{thm}{Theorem}[section]
\newtheorem{lemma}[thm]{Lemma}

\theoremstyle{remark}

\newcommand{\argmin}{\mathop{\mathrm{argmin}}}

\newcommand{\st}{\mathop{\mathrm{subject\,\,to}}}

\def\R{\mathbf{R}}

\allowdisplaybreaks

\begin{document}

\title{An ADMM Algorithm for a Generic $\ell_0$ Sparse Overlapping Group Lasso Problem}
\author{Youming Zhao\\
Email: \href{mailto:youming0.zhao@gmail.com}{youming0.zhao@gmail.com}
}

\date{First draft: March 31, 2022\quad Last update: \today}
\maketitle
\abstract{We present an alternating direction method of multipliers (ADMM) for a generic overlapping group lasso problem, 
where the groups can be overlapping in an arbitrary way. Meanwhile, we prove the lower bounds and upper bounds for both 
the $\ell_1$ sparse group lasso problem and the $\ell_0$ sparse group lasso problem. Also, we propose the algorithms for computing these bounds.}

\section{Generic $\ell_0$ sparse overlapping group lasso}
The generic $\ell_0$ sparse overlapping group lasso problem is defined as
\begin{equation}\label{L0-sparse-overlap-group-lasso}
\min_{x\in\R^n} \left \{F(x):=\frac{1}{2s}\left  \|x-v \right \|^2 +\lambda_0 \|x\|_0 + \lambda_1 \sum_{i=1}^m \|x_{G_i}\|_2\right \}
\end{equation}
where $m$ denotes the number of groups, and $G_i\subseteq \{1,2,\ldots,n\}$ contains the feature indices of the $i$-th group. Here $s$ is the step size employed to get $v$ based on $x$. Note that $\bigcap_{i=\{1,\ldots,m\}} G_i\neq\emptyset$. Now consider the problem,
\begin{equation}\label{ADMM-L0-overlap-group-lasso}
\begin{array}{ll}
\operatorname{minimize}_{x_{G_i}\in\R^{n_i},z\in\R^n}& \frac{1}{2s}\left  \|z-v \right \|^2 +\lambda_0 \|z\|_0 + \lambda_1 \sum_{i=1}^m \|x_{G_i}\|_2\\
\st & x_{G_i}-z_i=0,\quad i=1,2,\ldots,m
\end{array}
\end{equation}
where $z_i\in\R^{n_i}$ is defined by $(z_i)_j=z_{G(i,j)}=z_g$. Here, $G(i,j)=g$ denotes the global index (location) of the $j$-th element in the set (group) $G_i$. Hence, $z$ is a linear function of $z_{i}, i\in\{1,\ldots,m\}$.


The augmented Lagrangian for \eqref{ADMM-L0-overlap-group-lasso} is
\begin{equation}\label{augmented-ADMM-L0-overlap-group-lasso}
L_{\rho}(x_{G_i},z,y_{G_i})= \sum_{i=1}^m \left[\lambda_1\|x_{G_i}\|_2+ y_{G_i}^T(x_{G_i}-z_{i})+\frac{\rho}{2}\|x_{G_i}-z_{i}\|_2^2\right]+\frac{1}{2s}\left  \|z-v \right \|^2 +\lambda_0 \|z\|_0.
\end{equation}

\begin{align}
&x_{G_i}^{k+1}:=\argmin_{x_{G_i}} \left( \lambda_1  \|x_{G_i}\|_2+x_{G_i}^Ty^k_{G_i}+\frac{\rho}{2}\|x_{G_i}-z_{i}^{k+1}\|_2^2\right),\quad i=1,2,\ldots,m \\
&z^{k+1}:=\argmin_{z} \left(\frac{1}{2s}\left  \|z-v \right \|_2^2 +\lambda_0 \|z\|_0 +\sum_{i=1}^{m}\left(\frac{\rho}{2}\|x_{G_i}^{k+1}-z_{i}\|_2^2-z_{i}^Ty^k_{G_i}\right)\right) \\
&y^{k+1}_{G_i}:=y^k_{G_i} +\rho (x_{G_i}^{k+1}-z_{i}^{k+1})
\end{align}
where $x_{G_i}$ and $z$ are primal variables, and $y_{G_i}$ are dual variables. 
The $x$-update is actually a group lasso problem and can be solved with the proximal operator of group lasso as follows.
\begin{align*}
  x_{G_i}^{k+1}&=\argmin_{x_{G_i}} \left( \lambda_1  \|x_{G_i}\|_2+x_{G_i}^Ty^k_{G_i}+\frac{\rho}{2}\|x_{G_i}-z_{i}^{k+1}\|_2^2\right) \\
  &=\argmin_{x_{G_i}}\left(\frac{1}{2}\|x_{G_i}-z_{i}^{k+1}\|_2^2+x_{G_i}^T\frac{y^k_{G_i}}{\rho}+ \frac{\lambda_1}{\rho}\|x_{G_i}\|_2\right)  \\
  &=\argmin_{x_{G_i}}\left(\frac{1}{2}\|x_{G_i}-(z_{i}^{k+1}-\frac{y^k_{G_i}}{\rho})\|_2^2+ \frac{\lambda_1}{\rho}\|x_{G_i}\|_2\right) \\
  &=S_{\lambda_1/\rho}(z_{i}^{k+1}-\frac{y^k_{G_i}}{\rho}),\quad i=1,2,\ldots,m
\end{align*}
where $S_{\lambda}(\cdot)$ is a soft-thresholding operator defined as
\[
S_{\lambda}(a)=(\|a\|_2-\lambda)_+\frac{a}{\|a\|_2}.
\]

Now we derive the solution to the $z$-update.
\begin{dmath*}
z^{k+1}=\argmin_z\left(\frac{1}{2s}\left  \|z-v \right \|_2^2 +\lambda_0 \|z\|_0+\sum_{i=1}^{m}\left(\frac{\rho}{2}\|x_{G_i}^{k+1}-z_{i}\|_2^2-z_{i}^Ty^k_{G_i}\right) \right)
=\argmin_z\sum_{g=1}^{n}\left(\frac{1}{2s}(z_g-v_g)^2 +\lambda_0 \|z_g\|_0+\sum_{G(i,j)=g}\left(\frac{\rho}{2}(z_{i})_j^2-(y_i^k)_j(z_{i})_j -\rho(x_i^{k+1})_j(z_{i})_j\right)\right)
=\argmin_z\sum_{g=1}^{n}\left(\frac{1}{2s}(z_g-v_g)^2 +\lambda_0 \|z_g\|_0+\frac{k_g\rho}{2}z_g^2-z_g\sum_{G(i,j)=g}\left((y_i^k)_j +\rho(x_i^{k+1})_j\right)\right)
=\argmin_z\sum_{g=1}^{n}\left(\frac{1}{2s}z_g^2-\frac{1}{s}v_gz_g +\frac{k_g\rho}{2}z_g^2-z_g\sum_{G(i,j)=g}\left((y_i^k)_j +\rho(x_i^{k+1})_j\right)+\lambda_0 \|z_g\|_0\right)
=\argmin_z\sum_{g=1}^{n}\left((\frac{1}{2s}+\frac{k_g\rho}{2})z_g^2-\frac{v_g}{s}z_g -z_g\sum_{G(i,j)=g}\left((y_i^k)_j +\rho(x_i^{k+1})_j\right)+\lambda_0 \|z_g\|_0\right)
=\argmin_z\sum_{g=1}^{n}\left((\frac{1}{2s}+\frac{k_g\rho}{2})z_g^2-z_g\left(\frac{v_g}{s}+\sum_{G(i,j)=g}\left((y_i^k)_j +\rho(x_i^{k+1})_j\right)\right)+\lambda_0 \|z_g\|_0\right)
=\argmin_z\sum_{g=1}^{n}\left(\frac{1}{2}z_g^2-z_g\frac{v_g/s+\sum_{G(i,j)=g}\left((y_i^k)_j +\rho(x_i^{k+1})_j\right)}{1/s+k_g\rho}+\frac{\lambda_0}{1/s+k_g\rho} \|z_g\|_0\right)
\end{dmath*}
equivalently,
\begin{dmath*}
z_g^{k+1}=\argmin_{z_g} \left(\frac{1}{2}z_g^2-z_g\frac{v_g/s+\sum_{G(i,j)=g}\left((y_i^k)_j +\rho(x_i^{k+1})_j\right)}{1/s+k_g\rho}+\frac{\lambda_0}{1/s+k_g\rho} \|z_g\|_0\right)
=\argmin_{z_g} \left(\frac{1}{2}\left(z_g-\frac{v_g/s+\sum_{G(i,j)=g}\left((y_i^k)_j +\rho(x_i^{k+1})_j\right)}{1/s+k_g\rho}\right)^2+\frac{\lambda_0}{1/s+k_g\rho} \|z_g\|_0\right)
=H_{\sqrt{2\lambda_0/(1/s+k_g\rho)}}\left(\frac{v_g/s+\sum_{G(i,j)=g}\left((y_i^k)_j +\rho(x_i^{k+1})_j\right)}{1/s+k_g\rho}\right),\quad
\forall g\hiderel{\in}\{1,2,\ldots,n\}
\end{dmath*}
where $H_{\lambda}(\cdot)$ is a hard-thresholding operator defined as follows.
\begin{equation*}
H_{\lambda}(u)=\begin{cases}
                     u, & \mbox{if } |u|>\lambda \\
                     0, & \mbox{otherwise}.
                   \end{cases}
\end{equation*}
Note that when $u$ is a vector, $H_{\lambda}(\cdot)$ is an element-wise hard-thresholding operator.

Thus, we obtain the final update formulas for $x,z$ and $y$ as follows.
\begin{align*}
&x_{G_i}^{k+1}=S_{\lambda_1/\rho}(z_{i}^{k+1}-\frac{y^k_{G_i}}{\rho}),\quad \forall i\in \{1,2,\ldots,m\} \\
&z_g^{k+1}=H_{\sqrt{2\lambda_0/(1/s+k_g\rho)}}\left(\frac{v_g/s+\sum_{G(i,j)=g}\left((y_i^k)_j +\rho(x_i^{k+1})_j\right)}{1/s+k_g\rho}\right),\quad
\forall g\hiderel{\in}\{1,2,\ldots,n\}\\
&y^{k+1}_{G_i}=y^k_{G_i} +\rho (x_{G_i}^{k+1}-z_{i}^{k+1}),\quad \forall i\in \{1,2,\ldots,m\}
\end{align*}

\subsection{The matrix form of the $\ell_0$ sparse overlapping group lasso}
Define $\tilde{x}=[x^T_{G_1},x^T_{G_2},\ldots,x^T_{G_m}]^T\in\R^{\tilde{n}}$ and $\tilde{z}=[z^T_{G_1},z^T_{G_2},\ldots,z^T_{G_m}]^T\in\R^{\tilde{n}}$ where $\tilde{n}=\sum_{i=1}^{m}n_i$. Then $\tilde{z}$ can be represented as
\[
\tilde{z}=Gz
\]
where each row of $G$ has only one entry being $1$ and other entries being $0$. The corresponding definition of $z$-update becomes
\begin{dmath*}
z^{k+1}=\argmin_{z} \left(\frac{1}{2s}\left  \|z-v \right \|_2^2 +\lambda_0 \|z\|_0 +\frac{\rho}{2}\|\tilde{x}^k-Gz\|_2^2-(\tilde{y}^k)^{T}Gz\right)
=\argmin_{z} \left(\frac{1}{2}\left(\frac{z^Tz}{s}+\rho z^TG^TGz\right)-\left(\frac{v^Tz}{s} +(\rho\tilde{x}^k)^{T} Gz+(\tilde{y}^k)^{T}Gz\right) +\lambda_0 \|z\|_0 \right)
=\argmin_{z} \left(\frac{1}{2}z^T\left(\frac{I}{s}+\rho G^TG\right)z-\left(\frac{v^Tz}{s}+(\rho\tilde{x}^k)^{T}Gz+(\tilde{y}^k)^{T}Gz\right) +\lambda_0 \|z\|_0 \right)
=\argmin_{z} \left(\frac{1}{2}z^T\left(\frac{I}{s}+\rho G^TG\right)z-\left(\frac{v}{s}+G^T(\rho\tilde{x}^k+\tilde{y}^k)\right)^Tz +\lambda_0 \|z\|_0 \right)
\end{dmath*}
where $\tilde{y}=[y^T_{G_1},y^T_{G_2},\ldots,y^T_{G_m}]^T\in\R^{\tilde{n}}$. We observe that $G^TG$ is a diagonal matrix of which the $g$-th diagonal entry corresponds to the number of groups that the global variable $z_g$ involves. Since $G^TG$ is positive semidefinite and $s,\rho>0$, $I/s+\rho G^TG$ is definitely a positive definite matrix. Let $I/s+\rho G^TG=\operatorname{diag}(c_1,\ldots,c_n)$ and $C=\operatorname{diag}(\sqrt{c_1},\ldots,\sqrt{c_n})$ where $c_g=1/s+ k_g\rho$. Thus, $C^TC=I/s+\rho G^TG$. With this setting, we have
\begin{dmath*}
z^{k+1}=\argmin_{z} \left(\frac{1}{2}z^TC^TCz-\left(\frac{v}{s}+G^T(\rho\tilde{x}^k+\tilde{y}^k)\right)^TC^{-1}Cz +\lambda_0 \|z\|_0 \right)
=\argmin_{z} \left(\frac{1}{2}\|Cz-C^{-1}\left(\frac{v}{s}+G^T(\rho\tilde{x}^k+\tilde{y}^k)\right)\|_2^2 +\lambda_0 \|z\|_0 \right).
\end{dmath*}
Since $C$ is diagonal, the $z$-update reduces to $n$ subproblems as follows.
\begin{dmath*}
z_g^{k+1}=\argmin_{z_g} \left(\frac{1}{2}
\left(\sqrt{c_g}z_g-\frac{1}{\sqrt{c_g}}\cdot\left[\frac{v}{s}+G^T(\rho\tilde{x}^k+\tilde{y}^k)\right]_g\right)^2 +\lambda_0 \|z_g\|_0 \right)
=\argmin_{z_g} \left(\frac{1}{2}
\left(z_g-\frac{1}{c_g}\cdot\left[\frac{v}{s}+G^T(\rho\tilde{x}^k+\tilde{y}^k)\right]_g\right)^2 +\frac{\lambda_0}{c_g}\|z_g\|_0 \right)
=H_{\sqrt{2\lambda_0/c_g}}\left(\frac{1}{c_g}\cdot\left[\frac{v}{s}+G^T(\rho\tilde{x}^k+\tilde{y}^k)\right]_g\right)
\end{dmath*}
which is exactly the same as the previous counterpart result.

\subsection{Solving the dual problem via ADMM}
The dual of \eqref{ADMM-L0-overlap-group-lasso} is,
\begin{align*}
&\min_{x_{G_i},z}\frac{1}{2s}\left \|z-v \right \|^2 +\lambda_0 \|z\|_0 + \lambda_1 \sum_{i=1}^m \|x_{G_i}\|_2+ \sum_{i=1}^m y_{G_i}^T(x_{G_i}-z_{i})\\
=&\min_{z}\left(\frac{1}{2s}\left  \|z-v \right \|^2 +\lambda_0 \|z\|_0-\sum_{i=1}^m y_{G_i}^Tz_{i}\right)
+\min_{x_{G_i}} \left(\sum_{i=1}^m (\lambda_1\|x_{G_i}\|_2 + y_{G_i}^Tx_{G_i}) \right)\\
=&\min_{z}\left(\frac{1}{2s}(z^Tz-2v^Tz+v^Tv)-\tilde{y}^T\tilde{z} +\lambda_0 \|z\|_0\right)
+\sum_{i=1}^m\min_{x_{G_i}} \left( \lambda_1\|x_{G_i}\|_2 + y_{G_i}^Tx_{G_i} \right)\\
=&\min_{z}\left(\frac{1}{2s}(z^Tz-2v^Tz-2s\tilde{y}^TGz+v^Tv)+\lambda_0 \|z\|_0\right)
+\sum_{i=1}^m\min_{x_{G_i}} \left( \lambda_1\|x_{G_i}\|_2 + y_{G_i}^Tx_{G_i} \right)\\
=&\min_{z}\left(\frac{1}{2s}(z^Tz-2(v+sG^T\tilde{y})^Tz+v^Tv)+\lambda_0 \|z\|_0\right)
+\sum_{i=1}^m\min_{x_{G_i}} \left( \lambda_1\|x_{G_i}\|_2 + y_{G_i}^Tx_{G_i} \right)\\
=&\min_{z}\left(\frac{1}{2s}\left(\|z-(v+sG^T\tilde{y})\|^2+v^Tv-(v+sG^T\tilde{y})^T(v+sG^T\tilde{y})\right)+\lambda_0 \|z\|_0\right)
+\sum_{i=1}^m\min_{x_{G_i}} \left( \lambda_1\|x_{G_i}\|_2 + y_{G_i}^Tx_{G_i} \right)\\
=&\min_{z}\left(\frac{1}{2s}\|z-(v+sG^T\tilde{y})\|^2+\lambda_0 \|z\|_0\right)
+\sum_{i=1}^m\min_{x_{G_i}} \left( \lambda_1\|x_{G_i}\|_2 + y_{G_i}^Tx_{G_i} \right)+\frac{1}{2s}(\|v\|^2-\|v+sG^T\tilde{y}\|^2)\\
=&\min_{z}\left(\frac{1}{2s}\|z-(v+sG^T\tilde{y})\|^2+\lambda_0 \|z\|_0\right)
-\frac{1}{2s}\|v+sG^T\tilde{y}\|^2+\frac{1}{2s}\|v\|^2,\quad \|y_{G_i}\|_2\le \lambda_1, i=1,\ldots,m
\end{align*}
After dropping the constant term, the dual problem of \eqref{ADMM-L0-overlap-group-lasso} becomes
\begin{equation}\label{dual-problem-L0-group-lasso}
\displaystyle\max_{\tilde{y}\in\Omega}\min_{z\in\R^n} \left\{\psi(z,\tilde{y})=\frac{1}{2s}\|z-(v+sG^T\tilde{y})\|^2+\lambda_0 \|z\|_0
-\frac{1}{2s}\|v+sG^T\tilde{y}\|^2\right\}.
\end{equation}
where $\Omega$ is defined as follows:
\[
\Omega=\{\tilde{y}\in\R^{\tilde{n}}\ |\ \|y_{G_i}\|_2\le \lambda_1, i=1,2,\ldots,m\}.
\]
For a given $\tilde{y}^{k-1}$, the optimal $z$ minimizing $\psi(z,\tilde{y}^{k-1})$ in \eqref{dual-problem-L0-group-lasso} is given by
\begin{equation}\label{sol-z-dual-problem-L0-group-lasso}
z^k=H_{\sqrt{2s\lambda_0}}(v+sG^T\tilde{y}^{k-1}).
\end{equation}
Plugging \eqref{sol-z-dual-problem-L0-group-lasso} into \eqref{dual-problem-L0-group-lasso}, we get the following maximization problem with respect to $\tilde{y}$:
\begin{equation}\label{min-y-dual-problem-L0-group-lasso}
\max_{\tilde{y}\in\Omega}\ \{\omega(\tilde{y})=-\psi(z^k,\tilde{y})\}
\end{equation}
which is equivalent to the following problem
\begin{equation}
\max_{\tilde{y}\in\Omega}\ \tilde{y}^T G(z^k - 2v)
\end{equation}
which can be solved analytically as follows.
\begin{equation}
\tilde{y}_{G_i}=\frac{[G(z^k - 2v)]_i}{\|[G(z^k - 2v)]_i\|_2}
\end{equation}
where $[G(z^k - 2v)]_i\in\mathbf{R}^{n_i}$ denotes the counterpart corresponding to the group $G_i$. Finally, 
our methodology for minimizing the problem defined in \eqref{ADMM-L0-overlap-group-lasso} is to alternate update $z$ and $\tilde{y}$.




\subsection{The bounds on the optimal value of the overlapping group lasso}
Before presenting the results regarding the bounds of the optimal value of the $\ell_0$ sparse group lasso, we introduce three lemmas which lead to the upcoming theorem. For completeness, we describe a well known result as the following lemma, namely the quadratic mean (QM) is no less than the arithmetic mean (AM).
\begin{lemma}[QM$\ge$AM]\label{QM-AM}
Given $\mathbf{x}\in\R^n$, the following
\[
\sqrt{\frac{\sum_{i=1}^{n}x_i^2}{n}}\ge \frac{\sum_{i=1}^{n}|x_i|}{n}
\]
holds. The equality holds if and only if $x_1=x_2=\ldots=x_n$.
\end{lemma}
\begin{proof}
According to Cauchy-Schwartz inequality which says that given two vectors $\mathbf{x},\mathbf{y}\in\R^n$, $\|\mathbf{x}\|_2\|\mathbf{y}\|_2\ge |\mathbf{x}^T\mathbf{y}|$, we have
\begin{align*}
\sqrt{\frac{\sum_{i=1}^{n}x_i^2}{n}} &=\sqrt{\sum_{i=1}^{n}(\frac{|x_i|}{\sqrt{n}})^2}\cdot\overbrace{\sqrt{\sum_{i=1}^{n}(\frac{1}{\sqrt{n}})^2}}^{=1} \\
&\ge \frac{\sum_{i=1}^{n}|x_i|}{n}\ge \frac{\sum_{i=1}^{n}x_i}{n},
\end{align*}
where the equalities in the first and second inequalities hold if and only if $|x_1|=|x_2|=\cdots=|x_n|$ and $x_1=x_2=\cdots=x_n$, respectively. This completes the proof.
\end{proof}

\subsubsection{Lower bound on the overlapping group lasso}
\begin{lemma}[lower bound on the overlapping group lasso]
Given $\mathbf{x}\in\R^n$, $\mathbf{w}\in\R_{++}^m$ and some groups $G_i\subseteq \{1,2,\ldots,n\}, i=1,2,\ldots,m$, let $I_{G_i}(j)$ denote an indicator function whose value is $1$ if $j\in G_i$ and $0$ otherwise. Then, the following
\[
\|\mathbf{L}\mathbf{x}\|_1\le \sum_{i=1}^{m}w_i\|\mathbf{x}_{G_i}\|_2
\]
holds, where $\mathbf{L}=\operatorname{\mathbf{diag}}(\mathbf{l})$ with elements $l_j=\sum_{i=1}^{m}\frac{w_i}{\sqrt{|G_i|}}\odot \mathbb{I}(j\in G_i), j=1,\ldots,n$ and $\mathbb{I}(e)=1$ if $e$ is true, $0$ otherwise. The equality holds if and only if for every $G_i$, the entries of $\mathbf{x}_{G_i}$ are identical.
\end{lemma}
\begin{proof}
\begin{align*}
\sum_{i=1}^{m}w_i\|\mathbf{x}_{G_i}\|_2
&=w_1\|\mathbf{x}_{G_1}\|_2 + \cdots + w_m\|\mathbf{x}_{G_m}\|_2\\
&=\sum_{i=1}^{m} w_i\sqrt{\sum_{j=1}^{|G_i|}|x_j|^2}=\sum_{i=1}^{m} w_i\sqrt{|G_i|}\cdot\sqrt{\frac{\sum_{j=1}^{|G_i|}|x_j|^2}{|G_i|}}\\
&\ge \sum_{i=1}^{m}w_i\sqrt{|G_i|}\cdot \frac{\sum_{j=1}^{|G_i|}|x_j|}{|G_i|} =\sum_{i=1}^{m}\frac{w_i}{\sqrt{|G_i|}}\cdot\sum_{j=1}^{|G_i|}|x_j|\\
&=\sum_{j=1}^{n}\sum_{i=1}^{m}\left(\frac{w_i}{\sqrt{|G_i|}}\odot \mathbb{I}(j\in G_i)\right) |x_j|
\end{align*}
where $\mathbb{I}(e)$ is an indicator function defined as follows:
\[
\begin{cases}
  1, & \mbox{if $e$ is true} \\
  0, & \mbox{otherwise}.
\end{cases}
\]
The second line follows from the definition of $p$-norm ($p\ge 1$) and the third line from Lemma \ref{QM-AM}. The equality holds if and only if the entries belonging to the same group are identical for all the groups.
Let $\mathbf{L}=\operatorname{\mathbf{diag}}(\mathbf{l})$, the diagonal matrix with elements $l_j=\sum_{i=1}^{m}\frac{w_i}{\sqrt{|G_i|}}\odot \mathbb{I}(j\in G_i), j=1,\ldots,n$, so we have
\begin{equation}\label{lower-bound-L0-sparse-group-lasso}
\sum_{i=1}^{m}w_i\|\mathbf{x}_{G_i}\|_2\ge \|\mathbf{L}\mathbf{x}\|_1
\end{equation}
\end{proof}

\subsubsection{Computing the lower bound on the overlapping group lasso}
Since we have found the lower bound on the overlapping group lasso operator, the overlapping group lasso problem reduces to solving a weighted lasso problem as follows.
\begin{equation}\label{problem-lb-on-OGL}
\min_{\mathbf{x}\in\R^n}\left\{ f_{\text{GL\_lb}}(\mathbf{x}):= \frac{1}{2}\|\mathbf{x}-\mathbf{v}\|_2^2 + \lambda\|\mathbf{Lx}\|_1\right\}
\end{equation}
Since $f_{\text{GL\_lb}}(\mathbf{x})$ is separable w.r.t $\mathbf{x}$, this is equivalent to solving the following subproblem for each $i$.
\[
\min_{x_i\in\R}\left\{ f_{\text{GL\_lb}}(x_i):= \frac{1}{2}(x_i-v_i)^2 +\lambda l_i |x_i|\right\}
\]
If $x_i>0$, then $f_{\text{GL\_lb}}(x_i)= \frac{1}{2}(x_i-v_i)^2 +\lambda l_ix_i$. By the first-order optimality condition,
\[
\nabla f_{\text{GL\_lb}}(x_i)= x_i-v_i + \lambda l_i=0 \Longleftrightarrow x_i=v_i - \lambda l_i>0 \Longleftrightarrow v_i > \lambda l_i.
\]
For $x_i<0$, we have the following similar argument.
\[
\nabla f_{\text{GL\_lb}}(x_i)= x_i-v_i -\lambda l_i=0 \Longleftrightarrow x_i=v_i + \lambda l_i<0 \Longleftrightarrow v_i < -\lambda l_i.
\]
In the case of $x_i=0$, let $\nu$ be the subdifferential of $|x_i|$ at $x_i=0$, then $\nu\in[-1,1]$. Thus,
\[
0\in \partial f_{\text{GL\_lb}}(0)=0-v_i +\lambda l_i \nu \Longleftrightarrow \frac{v_i}{\lambda l_i}\in \nu \Longleftrightarrow |v_i|\le \lambda l_i
\]
where the RHS follows from the fact that $|\nu|\le 1$. To sum up, the solution to the subproblem is
\begin{equation}\label{soft-thresholding-operator-scalar}
x_i=\begin{cases}
v_i - \lambda l_i, & \mbox{if } v_i > \lambda l_i \\
0, & \mbox{if } |v_i|\le \lambda l_i \\
v_i + \lambda l_i, & \mbox{if }v_i < -\lambda l_i.
\end{cases}
\quad i=1,2,\ldots,n.
\end{equation}

\subsubsection{Upper bound on the overlapping group lasso}
\begin{lemma}[upper bound on the overlapping group lasso]
Given $\mathbf{x}\in\R^n$, $\mathbf{w}\in\R_{++}^m$ and some groups $G_i\subseteq \{1,2,\ldots,n\}, i=1,\ldots,m$, denote the total number of appearances in all groups by $k_j, j=1,\ldots,n$, and let $\mathbf{U}=\operatorname{\mathbf{diag}}(\sqrt{k_1}\|\mathbf{w}\|_2,\ldots,\sqrt{k_n}\|\mathbf{w}\|_2)$. Then, the following
\[
\sum_{i=1}^{m}w_i\|\mathbf{x}_{G_i}\|_2\le \|\mathbf{U}\mathbf{x}\|_2
\]
holds. The equality holds if and only if $\frac{\|\mathbf{x}_{G_1}\|_2}{w_1}=\ldots=\frac{\|\mathbf{x}_{G_m}\|_2}{w_m}$.
\end{lemma}
\begin{proof}
\begin{align*}
\sum_{i=1}^{m}w_i\|\mathbf{x}_{G_i}\|_2
&=w_1\|\mathbf{x}_{G_1}\|_2 + \cdots + w_m\|\mathbf{x}_{G_m}\|_2\\
&\le \sqrt{w_1^2+\cdots+w_m^2}\cdot \sqrt{\|\mathbf{x}_{G_1}\|_2^2+\cdots+\|\mathbf{x}_{G_m}\|_2^2}\\
&=\sqrt{\sum_{i=1}^{m} w_i^2}\cdot \sqrt{\sum_{g=1}^{n}k_gx_g^2}=\sqrt{\sum_{g=1}^{n}(\sum_{i=1}^{m} w_i^2)k_gx_g^2}
\end{align*}
where the second line follows from Cauchy-Schwarz inequality and the equality holds if and only if $\frac{\|\mathbf{x}_{G_1}\|_2}{w_1}=\ldots=\frac{\|\mathbf{x}_{G_m}\|_2}{w_m}$. Let $\mathbf{U}_0=\operatorname{\mathbf{diag}}(\sqrt{k_1},\ldots,\sqrt{k_n})$ and $\mathbf{w}=(w_1,\ldots,w_m)$, then we have
\begin{equation}\label{upper-bound-L0-sparse-group-lasso}
\sum_{i=1}^{m}w_i\|\mathbf{x}_{G_i}\|_2\le \|\mathbf{w}\|_2\cdot \sqrt{\mathbf{x}^T\operatorname{\mathbf{diag}}(k_1,\ldots,k_n)\mathbf{x}}=\|\mathbf{w}\|_2\cdot \|\mathbf{U}_0\mathbf{x}\|_2
\end{equation}
By the positive homogeneity of $\|\cdot\|_2$, $\|\mathbf{w}\|_2$ can be absorbed into $\mathbf{U}_0$ as follows.
\begin{equation}\label{upper-bound-L0-sparse-group-lasso-concise-form}
\sum_{i=1}^{m}w_i\|\mathbf{x}_{G_i}\|_2\le\|\mathbf{U}\mathbf{x}\|_2
\end{equation}
where $\mathbf{U}=\operatorname{\mathbf{diag}}(\mathbf{u})$ and $\mathbf{u}=(\sqrt{k_1}\|\mathbf{w}\|_2,\ldots,\sqrt{k_n}\|\mathbf{w}\|_2)$.
\end{proof}

\subsubsection{Computing the upper bound on the overlapping group lasso}
After replacing the overlapping group lasso operator with $\|\mathbf{U}\mathbf{x}\|_2$, the upper bound on the overlapping group lasso is equal to the optimal value of the following problem.
\begin{equation}\label{problem-ub-on-OGL}
\min_{\mathbf{x}\in\R^n}\left\{ f_{\text{GL\_ub}}(\mathbf{x}):= \frac{1}{2}\|\mathbf{x}-\mathbf{v}\|_2^2 + \lambda\|\mathbf{Ux}\|_2\right\}
\end{equation}
Let $\mathbf{s}$ be an element of the subdifferential of $\|\cdot\|_2$ at $\mathbf{0}$. Then if $\mathbf{x}=\mathbf{0}$, solving the following zero subgradient equation gives
\[
\mathbf{0}-\mathbf{v} + \lambda \mathbf{U}^T \mathbf{s}=\mathbf{0} \Longleftrightarrow \mathbf{v} = \lambda \mathbf{U}^T \mathbf{s}\Longleftrightarrow \mathbf{U}^{-1}\mathbf{v} = \lambda \mathbf{s} \Longleftrightarrow \|\mathbf{U}^{-1}\mathbf{v}\|_2\le \lambda
\]
Thus, we obtain that if $\|\mathbf{U}^{-1}\mathbf{v}\|_2\le \lambda$, the optimal minimizer is $\mathbf{0}$.

When $\mathbf{x}\neq \mathbf{0}$, we have
\[
\nabla f_{\text{GL\_ub}}(\mathbf{x})=\mathbf{x}-\mathbf{v} + \frac{\lambda\mathbf{U}^T\mathbf{Ux}}{\|\mathbf{Ux}\|_2}=0 \Longleftrightarrow \left(\mathbf{I}+\frac{\lambda\mathbf{U}^T\mathbf{U}}{\|\mathbf{Ux}\|_2}\right)\mathbf{x}=\mathbf{v}\Longleftrightarrow \mathbf{x}=\left(\mathbf{I}+\frac{\lambda\mathbf{U}^T\mathbf{U}}{\|\mathbf{Ux}\|_2}\right)^{-1}\mathbf{v}.
\]
Thus, we reformulate the optimality condition as $\mathbf{x}=T(\mathbf{x})$ for the case of $\mathbf{x}\neq \mathbf{0}$, where $T$ is the operator
\begin{equation}\label{fixed-operator-T}
T(\mathbf{x}):=\left(\mathbf{I}+\frac{\lambda\mathbf{U}^T\mathbf{U}}{\|\mathbf{Ux}\|_2}\right)^{-1}\mathbf{v}.
\end{equation}
Then we have the following result concerning $T(\mathbf{x})$.
\begin{thm}\label{fixed-point-thm}
$T(\mathbf{x})$ has a unique fixed point $\mathbf{x}^*$. In other words, the corresponding fixed point iteration
\[
\mathbf{x}^{(k+1)}:=T(\mathbf{x}^{(k)})
\]
converges to a unique $\mathbf{x}^*$.
\end{thm}
\begin{proof}
Since $\mathbf{U}$ is a diagonal matrix, $\mathbf{U}^T\mathbf{U}=\mathbf{U}^2$. Let $\mathbf{G}=\mathbf{I}+\frac{\lambda\mathbf{U}^T\mathbf{U}}{\|\mathbf{Ux}\|_2}$, then $\mathbf{G}$ is a diagonal matrix with elements $G_{ii}=1+\frac{\lambda u_i^2}{\|\mathbf{Ux}\|_2}, i=1,2,\ldots,n$. Thus, $\mathbf{G}^{-1}$ is also a diagonal matrix with elements
\begin{equation}\label{fixed-point-def-rho-i}
0<\rho_i=\frac{\|\mathbf{Ux}\|_2}{\|\mathbf{Ux}\|_2+\lambda u_i^2}=1-\frac{\lambda u_i^2}{\|\mathbf{Ux}\|_2+\lambda u_i^2}<1.
\end{equation}
Then we have $x_i=\rho_i v_i$ for each $i$. Let $\boldsymbol\rho$ be a vector whose $i$-th entry is $\rho_i$ with $\rho_i\in(0,1)$. So, $\mathbf{x}=\boldsymbol\rho\odot \mathbf{v}$ where $\odot$ is the element-wise Hadamard product, which indicates $\mathbf{x}$ is a contracted version of $\mathbf{v}$. Let $\rho_i(y)=1-\frac{\lambda u_i^2}{y+\lambda u_i^2}$ with $y=\|\mathbf{Ux}\|_2>0$. Then
\[
\rho_i^{\prime}(y)=\frac{\lambda u_i^2}{(y+\lambda u_i^2)^2}>0.
\]
which shows $\rho_i$ is a strictly increasing function of $y$, i.e., $\|\mathbf{Ux}\|_2$.
By \eqref{fixed-point-def-rho-i}, a smaller (bigger) $\|\mathbf{Ux}\|_2$ gives a smaller (bigger) $\rho_i$ for each $i$ which in turn generates smaller (bigger) $x_i$ via $x_i=\rho_i v_i$ for each $i$, and then small (greater) $\|\mathbf{Ux}\|_2$. Thanks to this interplay between $\|\mathbf{Ux}\|_2$ and $\boldsymbol\rho$, the sequences regarding $\|\mathbf{Ux}\|_2$ and $\boldsymbol\rho$ generated by performing $T(\mathbf{x})$ are monotone. For example, let us start the iteration with $\mathbf{x}^{(0)}\neq\mathbf{0}$. Then $\mathbf{x}^{(1)}=\boldsymbol\rho^{(0)}\odot \mathbf{v}$. Suppose $\|\mathbf{U}\mathbf{x}^{(0)}\|_2>\|\mathbf{U}\mathbf{x}^{(1)}\|_2$, then
\begin{equation}\label{fixed-point-rho}
\rho_i^{(0)}=1-\frac{\lambda u_i^2}{\|\mathbf{U}\mathbf{x}^{(0)}\|_2+\lambda u_i^2}>1-\frac{\lambda u_i^2}{\|\mathbf{U}\mathbf{x}^{(1)}\|_2+\lambda u_i^2}=\rho_i^{(1)},\quad i=1,2,\ldots,n.
\end{equation}
Thus,
\[
x_i^{(2)}=\rho_i^{(1)}v_i<\rho_i^{(0)}v_i=x_i^{(1)},\quad i=1,2,\ldots,n.
\]
Since $\mathbf{U}$ is a diagonal matrix with nonnegative diagonal entries $u_i$, we have
\[
\|\mathbf{U}\mathbf{x}^{(2)}\|_2=\sqrt{\sum_{i=1}^{n} (u_ix_i^{(2)})^2}<\sqrt{\sum_{i=1}^{n} (u_ix_i^{(1)})^2}=\|\mathbf{U}\mathbf{x}^{(1)}\|_2
\]
So, $\|\mathbf{U}\mathbf{x}^{(0)}\|_2>\|\mathbf{U}\mathbf{x}^{(1)}\|_2>\|\mathbf{U}\mathbf{x}^{(2)}\|_2$.
Substituting $\|\mathbf{U}\mathbf{x}^{(1)}\|_2$ and $\|\mathbf{U}\mathbf{x}^{(2)}\|_2$ into \eqref{fixed-point-rho}, we get $\rho_i^{(1)}>\rho_i^{(2)}>\rho_i^{(3)}$. By repeating this, the contraction interplay between $\|\mathbf{U}\mathbf{x}^{(k)}\|_2$ and $\rho_i{(k)}$ lead to that both $\{\|\mathbf{U}\mathbf{x}^{(k)}\|_2\}$ and $\{\rho_i^{(k)}\}_{i=1,2,\ldots,n}$ are decreasing sequences. Also, $\{\|\mathbf{U}\mathbf{x}^{(k)}\|_2\}$ and $\{\rho_i^{(k)}\}$ are both bounded below by $0$. Since monotone bounded sequences converge, $\{\|\mathbf{U}\mathbf{x}^{(k)}\|_2\}$ and $\{\rho_i^{(k)}\}$ are convergent. Assuming $\lim_{k\to \infty}\|\mathbf{U}\mathbf{x}^{(k)}\|_2=c$
and $\lim_{k\to \infty}\rho_i^{(k)}=\rho_i$
yield
\[
x_i=\frac{cv_i}{c+\lambda u_i^2},\quad i=1,2,\ldots,n.
\]
Multiplying both sides by $u_i$, squaring both sides and summing over $i$ gives
\begin{equation}\label{fixed-point-c}
c^2=\sum_{i=1}^{n}(u_ix_i)^2=\sum_{i=1}^{n}\frac{(cu_iv_i)^2}{(c+\lambda u_i^2)^2} \Longrightarrow 1= \sum_{i=1}^{n}\frac{(u_iv_i)^2}{(c+\lambda u_i^2)^2}
\end{equation}
The solution to the equation on the RHS of \eqref{fixed-point-c} is $c$ which is unique since $c=\|\mathbf{U}\mathbf{x}^*\|_2>0$. We can show this by contradiction. Specifically, suppose $c'>c$ is the solution to \eqref{fixed-point-c}. Since $\lambda,u_i>0, i=1,2,\ldots,n$, then we get
\[
\frac{(u_iv_i)^2}{(c'+\lambda u_i^2)^2}<\frac{(u_iv_i)^2}{(c+\lambda u_i^2)^2}\Longrightarrow\sum_{i=1}^{n}\frac{(u_iv_i)^2}{(c'+\lambda u_i^2)^2}<\sum_{i=1}^{n}\frac{(u_iv_i)^2}{(c+\lambda u_i^2)^2}=1
\]
which contradicts the supposition $c'>c$ is the solution to \eqref{fixed-point-c}, i.e., $\sum_{i=1}^{n}\frac{(u_iv_i)^2}{(c'+\lambda u_i)^2}=1$. Similar arguments hold for  the case when $c'<c$. Thus, $c$, i.e., $\|\mathbf{U}\mathbf{x}^*\|_2$ is unique. Furthermore, $\boldsymbol \rho$ is unique because of $\rho_i=1-\frac{\lambda u_i^2}{\|\mathbf{Ux}\|_2+\lambda u_i^2}, i=1,2,\ldots,n$.

For the case of $\|\mathbf{U}\mathbf{x}^{(0)}\|_2<\|\mathbf{U}\mathbf{x}^{(1)}\|_2$, similar arguments give increasing and bounded sequences $\{\|\mathbf{U}\mathbf{x}^{(k)}\|_2\}$ and $\{\rho_i^{(k)}\}$. Thus, they are convergent as well. If $\|\mathbf{U}\mathbf{x}^{(0)}\|_2=\|\mathbf{U}\mathbf{x}^{(1)}\|_2$, we luckily hit the fixed point in one step. Finally, $\mathbf{x}^*$ is unique due to $\mathbf{x}^*=\boldsymbol\rho\odot\mathbf{v}$. Therefore, $T(\mathbf{x})$ is a fixed point operator. This completes our proof.
\end{proof}

\subsection{The bounds on the optimal value of the $\ell_1$ sparse overlapping group lasso}
We have found the bounds for the overlapping group lasso operator in the previous section. Now it is natural to transform the bounds on the optimal value of the $\ell_1$ sparse overlapping group lasso into solving two problems.

\subsubsection{Lower bound on the optimal value of the $\ell_1$ sparse overlapping group lasso}
The lower bound can be obtained by solving the following problem.
\begin{equation}\label{problem-lb-on-L1-OGL}
\min_{\mathbf{x}\in\R^n}\left\{ f_{\ell_1\_\text{GL\_lb}}(\mathbf{x}):= \frac{1}{2}\|\mathbf{x}-\mathbf{v}\|_2^2 + \lambda\|\mathbf{Lx}\|_1+\lambda_1\|\mathbf{x}\|_1\right\}.
\end{equation}
Since $\mathbf{L}$ is a diagonal matrix, it can be rewritten as
\begin{equation}\label{problem-lb-on-L1-OGL-compact}
\min_{\mathbf{x}\in\R^n}\left\{ f_{\ell_1\_\text{GL\_lb}}(\mathbf{x}):= \frac{1}{2}\|\mathbf{x}-\mathbf{v}\|_2^2 + \|(\lambda\mathbf{L}+\lambda_1\mathbf{I})\mathbf{x}\|_1\right\}.
\end{equation}
which shares the same form as \eqref{problem-lb-on-OGL} and can be solved in a similar way. For brevity, we present its solution directly as follows.
\[
x_i=\begin{cases}
v_i - \lambda l_i, & \mbox{if } v_i > \lambda l_i + \lambda_1 \\
0, & \mbox{if } |v_i|\le \lambda l_i + \lambda_1 \\
v_i + \lambda l_i, & \mbox{if }v_i < -\lambda l_i- \lambda_1.
\end{cases}
\quad i=1,2,\ldots,n.
\]


\subsubsection{Upper bound on the optimal value of the $\ell_1$ sparse overlapping group lasso}
The upper bound can be obtained by solving the following problem.
\begin{equation}\label{problem-ub-on-L1-OGL}
\min_{\mathbf{x}\in\R^n}\left\{ f_{\ell_1\_\text{GL\_ub}}(\mathbf{x}):= \frac{1}{2}\|\mathbf{x}-\mathbf{v}\|_2^2 + \lambda\|\mathbf{Ux}\|_2+\lambda_1\|\mathbf{x}\|_1\right\}.
\end{equation}
Let $\boldsymbol\mu$ and $\boldsymbol\nu$ be the subdifferentials of $\|\mathbf{x}\|_2$ and $\|\mathbf{x}\|_1$ at $\mathbf{x}= \mathbf{0}$, where $\|\boldsymbol\mu\|_2\le 1$ and $\|\boldsymbol\nu\|_{\infty}\le1$. $\Delta$ and $\Lambda$ are defined as
\[
\Delta=\{\boldsymbol\mu\in\R^n\mid\|\boldsymbol\mu\|_2\le 1\},\quad \Lambda=\{\boldsymbol\nu\in\R^n\mid\|\boldsymbol\nu\|_{\infty}\le 1\}.
\]
If $\mathbf{x}=\mathbf{0}$, using the first-order optimality condition gives
\[
\mathbf{0}\in \mathbf{0}-\mathbf{v}+ \lambda \mathbf{U}^T\boldsymbol\mu+\lambda_1\boldsymbol\nu \Longleftrightarrow \mathbf{U}^{-1}\mathbf{v}\in\lambda \boldsymbol\mu+\lambda_1\mathbf{U}^{-1}\boldsymbol\nu \Longleftrightarrow \mathbf{U}^{-1}(\mathbf{v}-\lambda_1\boldsymbol\nu)=\lambda \boldsymbol\mu \Longleftrightarrow \max_{\boldsymbol\nu\in\Lambda}\,\|\mathbf{U}^{-1}(\mathbf{v}-\lambda_1\boldsymbol\nu)\|_2\le\lambda
\]
Since $\mathbf{U}^{-1}$ is also a diagonal matrix with positive diagonal entries, the maximum value of $\|\mathbf{U}^{-1}(\mathbf{v}-\lambda_1\boldsymbol\nu)\|_2$ is attained at $\boldsymbol\nu=-\operatorname{sgn}(\mathbf{v})\odot\mathbf{1}$. Here, $\operatorname{sgn}$ is an elementwise sign function whose value is $1$ for positive inputs, $-1$ for negative inputs and 0 otherwise, and $\mathbf{1}$ is a vector with all entries being $1$. With these settings, we have
\begin{equation*}
\max_{\boldsymbol\nu\in\Lambda}\,\|\mathbf{U}^{-1}(\mathbf{v}-\lambda_1\boldsymbol\nu)\|_2
=\|\mathbf{U}^{-1}\left(\mathbf{v}+\lambda_1\operatorname{sgn}(\mathbf{v})\odot\mathbf{1}\right)\|_2
=\sqrt{\sum_{i=1}^{n}\left(\frac{v_i+\lambda_1\operatorname{sgn}(v_i)}{u_i}\right)^2}\le \lambda
\end{equation*}
Hence, we get that $\mathbf{x}=\mathbf{0}$ if and only if $\|\mathbf{U}^{-1}\left(\mathbf{v}+\lambda_1\operatorname{sgn}(\mathbf{v})\odot\mathbf{1}\right)\|_2\le \lambda$.

Now we talk about the case of $\mathbf{x}\neq\mathbf{0}$. By the first-order optimality condition, we have
\[
\mathbf{0}\in \mathbf{x}-\mathbf{v}+ \lambda \frac{\mathbf{U}^T\mathbf{Ux}}{\|\mathbf{Ux}\|_2}+\lambda_1\mathbf{x}'
\]
where $\mathbf{x}'$ denotes the subdifferential of $\|\mathbf{x}\|$ at $\mathbf{x}\neq\mathbf{0}$ defined as follows.
\[x'_i=
\begin{cases}
  1, & \mbox{if } x_i>0 \\
  -1, & \mbox{if } x_i<0 \\
  \nu\in[-1,1], & \mbox{if } x_i=0.
\end{cases}
\]
If $x_i=0$, we have
\[
x_i=0\Longleftrightarrow 0\in 0-v_i+0+\lambda_1\nu \Longleftrightarrow v_i\in\lambda_1\nu\Longleftrightarrow |v_i|\le \lambda_1.
\]
By contradiction, if $x_i\neq0$, it is clear to see that the optimal $x_i$ shares the common sign with $v_i$, otherwise it will lead to greater objective values. Thus, we get
\[
\mathbf{\tilde{x}}+ \lambda \frac{\mathbf{\tilde{U}}^T\mathbf{\tilde{U}\tilde{x}}}{\|\mathbf{\tilde{U}\tilde{x}}\|_2}
=\mathbf{\tilde{v}}-\lambda_1\operatorname{sgn}(\mathbf{\tilde{v}})\odot\mathbf{1}
\Longleftrightarrow \mathbf{\tilde{x}}=T_{\ell_1\text{\_ub}}\left(\mathbf{I}+\lambda \frac{\mathbf{\tilde{U}}^T\mathbf{\tilde{U}\tilde{x}}}{\|\mathbf{\tilde{U}\tilde{x}}\|_2} \right)^{-1}(\mathbf{\tilde{v}}-\lambda_1\operatorname{sgn}(\mathbf{\tilde{v}})\odot\mathbf{1})
\]
where $\mathbf{\tilde{x}}$ represents the reduced $\mathbf{x}$ after removing zero entries, and $\mathbf{\tilde{v}},\mathbf{\tilde{U}}$ are the corresponding notations. By Theorem \ref{fixed-point-thm}, $T_{\ell_1\text{\_ub}}(\mathbf{\tilde{x}})$ is a fixed point operator.

\subsection{The bounds on the optimal value of the $\ell_0$ sparse overlapping group lasso}
\subsubsection{Lower bound on the optimal value of the $\ell_0$ sparse overlapping group lasso}
The lower bound can be obtained by solving the following problem.
\begin{equation}\label{problem-lb-on-L0-OGL-subp}
\min_{\mathbf{x}\in\R^n}\left\{ f_{\ell_0\_\text{GL\_lb}}(\mathbf{x}):= \frac{1}{2}\|\mathbf{x}-\mathbf{v}\|_2^2 + \lambda\|\mathbf{Lx}\|_1+\lambda_0\|\mathbf{x}\|_0\right\}.
\end{equation}
which is separable and can be divided into subproblems as follows.
\begin{equation}\label{problem-lb-on-L0-OGL-subp}
\min_{x_i\in\R}\left\{ f_{\ell_0\_\text{GL\_lb}}(x_i):= \frac{1}{2}(x_i-v_i)^2 + \lambda l_i\|x_i\|_1+\lambda_0\|x_i\|_0\right\}.
\end{equation}

If $x_i\neq0$, then \eqref{problem-lb-on-L0-OGL-subp} can be reduced to solving the following simpler problem.
\[
\min_{x_i\in\R} f_{\ell_0\_\text{GL\_lb}}(x_i)=\frac{1}{2}(x_i-v_i)^2 + \lambda l_i|x_i|+\lambda_0
\]
whose solution is given by \eqref{soft-thresholding-operator-scalar}, namely, $x_i=v_i-\lambda l_i\operatorname{sgn}(v_i)$. In this case, if $|v_i|>\lambda l_i$, the corresponding objective value is
\[
f_{\ell_0\_\text{GL\_lb}}(v_i-\lambda l_i\operatorname{sgn}(v_i))=\frac{1}{2}(\lambda l_i)^2 + \lambda l_i|v_i-\lambda l_i\operatorname{sgn}(v_i)|+\lambda_0,
\]
and if $|v_i|\le\lambda l_i$, we are done and definitely $x_i=0$. However, in the case of $|v_i|>\lambda l_i$, we still need to compare $\frac{1}{2}(\lambda l_i)^2 + \lambda l_i|v_i-\lambda l_i\operatorname{sgn}(v_i)|+\lambda_0$ with $f_{\ell_0\_\text{GL\_lb}}(0)=\frac{1}{2}v_i^2$ due to the existence of the additional term $\lambda_0$. If $f_{\ell_0\_\text{GL\_lb}}(0)\le f_{\ell_0\_\text{GL\_lb}}(v_i-\lambda l_i\operatorname{sgn}(v_i))$, the solution is $0$ rather than $v_i-\lambda l_i\operatorname{sgn}(v_i)$.

\subsubsection{Upper bound on the optimal value of the $\ell_0$ sparse overlapping group lasso}
The upper bound can be obtained by solving the following problem.
\begin{equation}\label{problem-ub-on-L0-OGL}
\min_{\mathbf{x}\in\R^n}\left\{ f_{\ell_0\_\text{GL\_ub}}(\mathbf{x}):= \frac{1}{2}\|\mathbf{x}-\mathbf{v}\|_2^2 + \lambda\|\mathbf{Ux}\|_2+\lambda_0\|\mathbf{x}\|_0\right\}.
\end{equation}
By the definition of the induced norm of $\|\cdot\|_{a,b}$, we have $\|\mathbf{Ux}\|_2\le \|\mathbf{U}\|_2\|\mathbf{x}\|_2$. When $a=b=2$, $\|\mathbf{U}\|_2$ is called the spectral norm and it is equal to the maximum singular value of $\mathbf{U}$, denoted as $\sigma_{\mathrm{max}}(\mathbf{U})$. Thus, the upper bound can be relaxed as follows.
\begin{equation}\label{problem-ub-on-L0-OGL-relaxed}
\min_{\mathbf{x}\in\R^n}\left\{ f_{\ell_0\_\text{GL\_ub}}(\mathbf{x}):= \frac{1}{2}\|\mathbf{x}-\mathbf{v}\|_2^2 + \lambda\sigma_{\mathrm{max}}(\mathbf{U})\|\mathbf{x}\|_2+\lambda_0\|\mathbf{x}\|_0\right\}.
\end{equation}
which has a closed-form solution proposed by \cite{DBLP:journals/sensors/ShaoZCPPLWM22}.

\section{Acknowledgement}
The inspiration of this work is from reading \cite{Boyd2011DistributedOA}.

\bibliographystyle{apalike}
\bibliography{mybib}

\end{document}